\documentclass[11pt]{amsart}

\usepackage{wrapfig}

\usepackage[letterpaper,margin=1in]{geometry}

\usepackage{amssymb,amsmath,graphicx,amsthm,mathrsfs,amsbsy}
\usepackage{hyperref}
\usepackage{algorithm}
\usepackage{algorithmicx}
\usepackage{algpseudocode}
\usepackage{color}

\usepackage{subcaption}
\usepackage{tikz}

\usepackage{amscd,verbatim}
\usepackage{enumerate}
\usepackage{bm}

\graphicspath{{Figures/}{Figures/images/}}

\usetikzlibrary{
  decorations.pathmorphing,
  decorations.pathreplacing
}

\DeclareMathOperator*{\argmin}{arg\,min}

\newtheorem{theorem}{Theorem}[section]

\newtheorem{lemma}[theorem]{Lemma}

\newtheorem{assumption}[theorem]{Assumption}

\def\c1{\mathbf 1}

\newcommand{\R}[1]{{\rm I\! R}^{#1}}

\begin{document}

\title{Constructing Frequency Domains on Graphs in Near-Linear Time}
\author{John C. Urschel\textsuperscript{a}}
\address{\textsuperscript{a}Department of Mathematics, Massachusetts Institute of Technology,
  Cambridge, MA, USA.}
 \author{Wenfang Xu\textsuperscript{b}}
\author{Ludmil T. Zikatanov\textsuperscript{b,c,$\ast$}}\thanks{$^\ast$Corresponding author. Email: ludmil@psu.edu}
\address{
  \textsuperscript{b}Department of Mathematics, Penn State University,
  University Park, PA, USA.}
\address{
  \textsuperscript{c}Institute for Mathematics and Informatics, Bulgarian Academy of Sciences, Sofia, Bulgaria.}

\maketitle

\begin{abstract}
  Analysis of big data has become an increasingly relevant area of
  research, with data often represented on discrete networks both
  constructed and organic. While for structured domains, there exist
  intuitive definitions of signals and frequencies, the definitions
  are much less obvious for data sets associated with a given
  network. Often, the eigenvectors of an induced graph Laplacian are
  used to construct an orthogonal set of low-frequency vectors. For
  larger graphs, however, the computational cost of creating such
  structures becomes untenable, and the quality of the approximation
  is adequate only for signals near the span of the set. We propose a
  construction of a full basis of frequencies with computational
  complexity that is near-linear in time and linear in storage. Using
  this frequency domain, we can compress data sets on unstructured
  graphs more robustly and accurately than spectral-based
  constructions.
\end{abstract}



\section{Introduction}

Graphs provide a representation for data of interacting discrete
objects in a given system. For structured data, such as an image or
video, there exist extremely efficient methods for compression and
storage, such as the fast Fourier transform \cite{1978WinogradS-aa}
and wavelet approaches \cite{ref1}. These techniques are often hailed
as some of the most important and influential discoveries of the
$20^{\text{th}}$ century. When dealing with unstructured data or a
discrete graph (which lacks geometric foundation), the definition of a
frequency becomes much less clear and the analogous techniques less
obvious. This has made the area of signal processing on graphs a major
research area in the $21^{\text{st}}$ century thus far.

Sometimes, the graph is a constructed representation, such as a
``similarity'' graph \cite{Wu:1993:OGT:628307.628538}. This
construction is useful in areas such as spectral clustering
\cite{Ng01onspectral,Luxburg:2007:TSC:1288822.1288832}, data
visualization \cite{Battista:1998:GDA:551884}, and dimension reduction
of data \cite{Coifman20065,Belkin:2003:LED:795523.795528}. However,
often the graph is not a construct. Instead, the natural structure
from which the data comes is a pre-existing network, such as a
biological \cite{Bu03topologicalstructure,cite-key12}, economic
\cite{jackson2008social}, neurological \cite{cite-key}, transportation
\cite{bell1997transportation}, sensor \cite{olfati2005consensus}, or
social \cite{newman2006finding,kleinberg1999authoritative} network. In
what follows, we may safely ignore whether the given graph is organic
or constructed. In both cases, the connectivity of the graph
represents information regarding the similarity of discrete objects
and, therefore, implies similarity and smoothness of the data produced
by related objects.

Most often, an associated discrete Laplacian of the graph is used, and
the spectrum and eigenvectors of the Laplacian are used to define a
frequency domain, with the eigenspaces of minimal and maximal
eigenvalues producing low- and high-pass filters, respectively
\cite{6808520}. This definition is intuitive; it is well known that
the Fourier transform of a real-valued function on the real line is an
expansion in the eigenfunctions of the continuous Laplacian. There is
great debate as to which discrete Laplacian is the most suitable; the
answer varies depending on which properties are deemed most important
\cite{4613}. In this work, we consider the unnormalized graph
Laplacian of $G = (V,E)$, denoted by $L(G) \in \R{n\times n}$,
$n=|V|$, and defined by the bilinear form
\begin{equation} \label{eqn1}
\langle L(G) u,v \rangle = \sum_{(i,j)\in E} (u_i-u_j)(v_i-v_j)
\end{equation}
for any $u,v \in \R{n}$, though other choices, such as the normalized
or random walk Laplacian, would be equally appropriate. These
spectral-type frequency constructions have been studied extensively in
\cite{DBLP:journals/tsp/SandryhailaM13,6879640,shuman_ACHA_2013,shuman_TSP_2013,ricaud_SPIE_2013},
among others. Other techniques exist, such as wavelets on graphs \cite{rustamov2013wavelets,hammond2011wavelets,gavish2010multiscale,ram2011generalized,tremblay2014graph}, graph filterbanks \cite{narang2013compact}, or aggregation-based algorithms \cite{tremblay2016subgraph}. In addition, the application of many discrete signal processing concepts to graphs was studied in \cite{ sandryhaila2013discrete,sandryhaila2014discrete,sandryhaila2014big,chen2015discrete}. For an excellent coverage of the current innovations and
projected areas of future work and research for signal processing on graphs, we refer the reader to
\cite{shuman_SPM_2013}.

In what follows, we use the localization properties of the graph to
construct a multilevel structure of aggregations and produce a
complete frequency domain with complexity that is near-linear in time
and linear in storage, the first aggregation-based algorithm in the literature to do so. Our technique allows us to construct frequencies on large networks when the complexity of other methods may
not be feasible. In addition, the efficient computation of a full
basis of frequencies allows a much more robust representation of
signals and does not rely heavily on the assumption of exclusively
low-frequency data.

\section{Near-Linear Time Frequency Domain Construction} \label{sec2}

Let $G=(V,E)$, $V = \{1,...,n\}$, be a given simple, connected, and undirected graph. For
each vertex $i\in V$, we define its neighborhood $N(i)$ and degree
$d(i)$ as
\[
N(i) = \{j\in V\;|\; (i,j)\in E\}, \quad 
d(i) = |N(i)|. 
\] 
For $w\in \R{n}$ and
  $\sigma \subset \{1,\ldots,n\}$, $| \sigma | = k$, by $w_\sigma \in \R{n}$ and $[w]_\sigma \in \R{k}$, we denote the restrictions of $w$ to $\sigma$ in $\R{n}$ and $\R{k}$, respectively. We denote the subgraph of $G$ induced by vertices $U \subset V$ by $G[U]$. In what follows, the $\ell^2$-inner product is denoted by $\langle \cdot , \cdot \rangle$ and the standard (Euclidean) norm induced by the $\ell^2$-inner product is denoted by $\| \cdot \|$. An aggregation $\mathcal{A} = \{a_1,...,a_{|\mathcal{A}|} \}$ of a graph $G = (V,E)$ is a non-overlapping partition of $V$ into connected components, namely,
\begin{enumerate}
\item $V = \bigcup_{a\in\mathcal{A}} a$,
\item $a\cap a' = \emptyset$, 
$a,a'\in\mathcal{A}$, $a \ne a'$,
\item $G[a]$ connected, $a \in \mathcal{A}$.
\end{enumerate}

 Here, and in what follows, we will consider only aggregations $\mathcal{A}$ such that $|a|\ge 2$ for all $a \in \mathcal{A}$. We can construct such a set using a randomized augmented matching
algorithm, such as Algorithm \ref{alg1}. Other suitable aggregation algorithms are found
in~\cite{Rajagopalan06dataaggregation}. When constructing an aggregation, we aim to minimize both the number of aggregates of order greater than two and the largest order of any aggregate, while using an algorithm that has complexity at worst near-linear in the size of the graph. Intuitively, for the majority of graphs, we should be able to accomplish both of these goals reasonably well. In particular, we recall a corollary of a famous result of Erdos and Renyi.

\begin{theorem}[\cite{erdos1966existence}] \label{thm:er}
Let $G_{n,m}$ be the graph with vertex set $V = \{1,...,n\}$ and edge set $E$ a random $m$-subset of $\{ (i,j) | i, j \in V, i \ne j \}$. If $m \ge (\frac{1}{2}+C) n \log n$ for some fixed constant $C$, then
$$ \lim_{\tiny \begin{array}{c} n \rightarrow \infty \\ n \text{ even}\end{array}} \mathrm{P} [ G_{n,m} \text{ has a perfect matching }] = 1.$$
\end{theorem}

All the experiments performed will use Algorithm \ref{alg1}, but we note that there does exist an algorithm for finding an approximately maximum matching in near-linear time \cite{duan2010approximating}.

\begin{algorithm}
\caption{Construct Aggregation $\mathcal{A}$ \label{alg1}}
\begin{enumerate}
\item[]
{ \bf Set:} ${\mathcal{A}}=\emptyset$, 
$\mathcal{S}=\emptyset$,   $\mathcal{R} = V$.
\item[] {\bf For} $i\in \mathcal{R}$, set $\mathcal{R}_i=\mathcal{R} \cap N(i)$.
  \begin{enumerate}
  \item[] {\bf If} $\mathcal{R}_i \ne \emptyset$, {\bf then}
    \begin{enumerate}
    \item[] \(j = \arg\min \{d(\hat j) \; | \; \hat j \in \mathcal{R}_i \};\)
    \item[] update: $\mathcal{A}\leftarrow \mathcal{A} \cup \{i,j\}$,
      $\mathcal{R} \leftarrow \mathcal{R} \setminus \{i,j\}$,   \\ $\text{\qquad \quad}$
      $d(k)\leftarrow d(k)-1$, $k\in N(i)$,\\ $\text{\qquad \quad}$ $d(\ell)\leftarrow d(\ell)-1$, $\ell \in N(j)$;
    \end{enumerate}
  \item[] {\bf else} update: $\mathcal{S}\leftarrow \mathcal{S} \cup \{i\}$, 
      $\mathcal R \leftarrow \mathcal R \setminus \{i \}$.

  \end{enumerate} 
\item[] {\bf For} $i \in \mathcal{S}$, 
  \begin{enumerate}
  \item[] $a = \arg\min\{|b|\; | \;  b \in \mathcal{A}, b\cap N(i) \neq \emptyset\}$, \\ $a\leftarrow a\cup \{i\}$.
  \end{enumerate}
\end{enumerate}
\end{algorithm}

For a given aggregation $\mathcal{A}$, we can define an associated
graph on the aggregates, namely,
$G_\mathcal{A} = (V_\mathcal{A}, E_\mathcal{A})$, with
$$ V_\mathcal{A} = \{ 1 , ... , |\mathcal{A}| \},$$
$$E_\mathcal{A} = \{ (i,j) \; | \; \exists \; u \in a_i, v \in a_j, (u,v) \in E \}.$$

For each aggregate $a_i \in \mathcal{A}$, consider the
eigendecomposition
\begin{equation} \label{eqn2}
  L(G[a_i]) =  Q_{a_i} \Lambda_{a_i} Q_{a_i}^*,
\end{equation}
with $\Lambda_{a_i}(j,j) \le \Lambda_{a_i}(k,k)$, $j \le k$, and
$Q_{a_i}$ unitary. Let $Q_{\mathcal{A}}\in \R{n\times n}$ be the unitary matrix acting on $G$ whose restriction on $a_i$
equals $Q_{a_i}$, namely,
\[
[Q_{\mathcal{A}} u ]_{a_i} = Q_{a_i}^* [ u ]_{a_i}, \quad \mbox{for all}\quad u\in \R{n}, 
\quad a_i \in \mathcal{A}. 
\] 
We define a partition of the set of vertices
$V = V_+ \cup V_- \cup V_*$, where
\begin{eqnarray*}
V_+ &=& \{ i \in V | \; \exists \; j \text{ s.t. } a_j(1) = i \},\\
V_- &=& \{ i \in V | \; \exists \; j \text{ s.t. } a_j(2) = i \},\\
V_* &=&  V \setminus [ V_+ \cup V_-].
\end{eqnarray*}
Therefore, for a permutation 
\(\pi:\{1:n\}\mapsto \{1:n\}\) and any $u\in \R{n}$, we have
\[
\pi Q_{\mathcal{A}} u = \begin{bmatrix} u_+ \\ u_- \\
  u_* \end{bmatrix}, \; 
u_{\pm} = [ Q_{\mathcal{A}} u]_{{V}_\pm}, \;  u_* = [Q_{\mathcal{A}} u]_{{V}_*}.
\]
Each entry in $u_+$ and $u_-$ is associated with a vertex
of the aggregation graph $G_{\mathcal{A}}$. We then have two copies of
$G_{\mathcal{A}}$, one with vertex set $V_{+}$ and the other with
  vertex set $V_-$, which allows us to repeat our aggregation procedure for the graph $G_{\mathcal{A}}$.
  
  When $|a_i| =2$, which, by Theorem \ref{thm:er} and
  inspection of Algorithm \ref{alg1}, occurs the majority of the time for well-behaved graphs,
  we have
\[ 
Q_{a} = \frac{1}{\sqrt{2}} \begin{bmatrix} 1 & 1 \\ 1 & -1  \end{bmatrix}.
\]
It follows that, for a given aggregate
$\mathcal{A}\ni a = \{ i_1 , i_2 \}$, the vectors $u_+$ and $u_-$
restricted to $a$ are given by the normalized sum and difference of
the elements:
\begin{eqnarray*}
\sqrt{2} \; [u_+ ]_{a} &=& u_{i_1} + u_{i_2},\\
\sqrt{2} \; [u_- ]_{a} &=& u_{i_1} - u_{i_2}.
\end{eqnarray*}

We can apply this procedure
recursively until we have
$|\mathcal{A}_{J}|=1$ for some $J$.  This results in a multilevel structure of graphs
$G_{\mathcal{A}_j}$, aggregations $\mathcal{A}_j$, unitary matrices $Q_{\mathcal{A}_j}$, and permutations $\pi_j:\{1:n_j\}\mapsto \{1:n_j\}$, for $j = 0,...,J$, where
$n_j=|\mathcal{A}_{j-1}|$ is the number of aggregates on level $j-1$, $n_0 := n$, and
each $Q_{\mathcal{A}_j}$ is an $n_j\times n_j$ matrix. For an illustration of this multilevel structure for a structural discretization of a Comanche helicopter, see Figure \ref{figheli}.

\begin{figure}[]
  \begin{subfigure}{0.32\textwidth}
    {\includegraphics[width = 1.65in, height = 1 in]{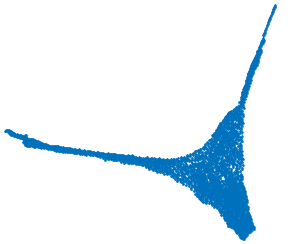}} \quad
    \caption{Original Graph $G$}
  \end{subfigure}
  \begin{subfigure}{0.32\textwidth}
    {\includegraphics[width = 1.65in,height = 1in]{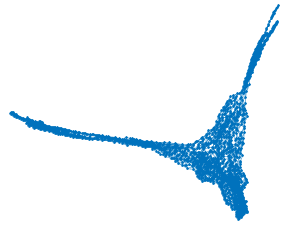}}
    \caption{Aggregation Graph $G_{\mathcal{A}_0}$}
\end{subfigure}
\begin{subfigure}{0.32\textwidth}
  {\includegraphics[width = 1.65in, height = 1in]{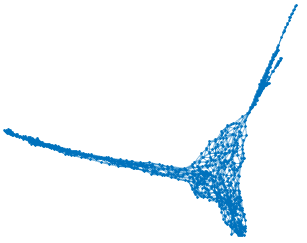}}
  \caption{Aggregation Graph $G_{\mathcal{A}_1}$}
\end{subfigure}
\begin{subfigure}{0.32\textwidth}
  {\includegraphics[width = 1.65in, height = 1in]{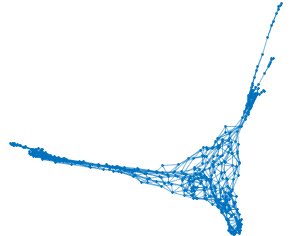}}
  \caption{Aggregation Graph $G_{\mathcal{A}_2}$}
\end{subfigure}
\begin{subfigure}{0.32\textwidth}
  {\includegraphics[width = 1.65in, height = 1in]{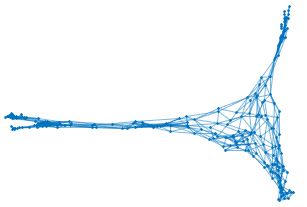}}
  \caption{Aggregation Graph $G_{\mathcal{A}_3}$}
\end{subfigure}
\begin{subfigure}{0.32\textwidth}
  {\includegraphics[width = 1.65in, height = 1in]{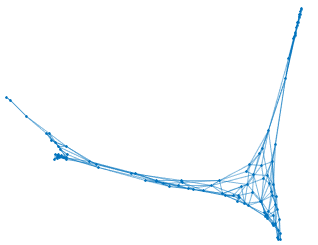}}
  \caption{Aggregation Graph $G_{\mathcal{A}_4}$}
\end{subfigure}
\caption{Successive application of random aggregation Algorithm \ref{alg1} to a structural discretization of a Comanche helicopter. Source \cite{Davis:2011:UFS:2049662.2049663}.}
\label{figheli}
\end{figure}

Let us define
\[
\mathcal{Q}_j' = 
\operatorname{diag}(\underbrace{\pi_j Q_{\mathcal{A}_j},\ldots, \pi_j Q_{\mathcal{A}_j}}_{\textstyle 2^j},I_j).
\] 
The matrix $I_j$ is the $m_j\times m_j$ identity matrix, where
$m_0=0$, and $m_{j} = m_{j-1} + n_{j-1} - 2 n_{j}$, for
$j=1,\ldots,J$. Let $\Pi_j\colon \{1:n\} \to \{1:n\}$ be a permutation
such that
\begin{equation*}
  \Pi_j
  \begin{bmatrix}
    u_{1,+} \\ u_{1,-} \\ u_{1,*} \\ \vdots \\ \vdots \\ u_{2^j,+} \\
    u_{2^j,-} \\ u_{2^j,*} \\ u_*
  \end{bmatrix} =
  \begin{bmatrix}
    u_{1,+} \\ u_{1,-} \\ \vdots \\ u_{2^j,+} \\ u_{2^j,-} \\
    u_{2^j,*} \\ \vdots \\ u_{1,*} \\ u_*
  \end{bmatrix}, \quad \parbox{.38\linewidth}{for
    $u_{i,\pm}\in\mathbb{R}^{n_{j+1}}$,
    $u_{i,*}\in\mathbb{R}^{n_j-2n_{j+1}}$, $i =
    1,\dots,2^j$, and $u_*\in\mathbb{R}^{m_j}$.}
\end{equation*}
We then define
\begin{equation*}
  \mathcal Q_j = \Pi_j\mathcal{Q}_j', \;
  \mathcal Q = \prod_{j=0}^J \mathcal{Q}_{J-j}.
\end{equation*}
By the properties of unitary matrices, $\mathcal{Q}$ is invertible,
and its inverse is
\[
\mathcal{Q}^{-1}=\mathcal Q^* = \prod_{j=0}^J \mathcal{Q}^*_{j}.
\]  
The columns of $\mathcal{Q^*}$ form an orthogonal basis
of $\R{n}$, referred to as a basis in the frequency domain
or a wavelet basis.  The rationale given above is summarized in Algorithm~\ref{alg2}.

\begin{algorithm}
	\caption{Make Frequency Basis and Filter Signal\label{alg2}}
  \begin{enumerate}
  \item Construct aggregates and unitary matrices
    $\{ \mathcal{A}_i , \pi_i Q_{\mathcal{A}_i} \}_{i=0}^J$. Store
    $\{ \pi_i Q_{\mathcal{A}_i} \}_{i=0}^J$.
  \item Given signal $u\in \R{n}$, define
    $\widehat u = \prod_{i=0}^J \mathcal {Q}_{J-j} u$ as the
    frequency representation of $u$.
  \item Given $k \in \{1,...,n\}$, let $$\sigma_k = \arg\max \{ \| \widehat u_{\sigma} \| \; | \; | \sigma | = k \}.$$ Define the $k$-frequency filter of $u$ as 
\( v = \mathcal{Q}^* \widehat{u}_{\sigma_k}\).
  \end{enumerate}
\end{algorithm}

In practice, the eigendecomposition in equation (\ref{eqn2}) is not computed explicitly. Instead, a dictionary of
eigendecompositions for low-order graphs (e.g. $|V| \le 5$) is stored a
priori. The lookup time for this dictionary is an $O(1)$
procedure. Because the aggregate size is rarely greater than three for well-behaved graphs,
this results in significantly less computation. For the remainder of the paper, we will make the following assumption:

\begin{assumption}
There exists some constant $C_{\mathcal{A}}$, independent of $n$, such that 
$$\max_{j=0,...,J} \; \max_{a \in \mathcal{A}_j} \; |a| \le C_{\mathcal{A}}.$$
\end{assumption}

Of course, there exists examples for which this is not the case (e.g. the star graph $S_n$), but Theorem \ref{thm:er} and the experiments in Section \ref{sec:num} support this assumption in practice.

Once the multilevel structure has been constructed, the aggregations can be
discarded and only the sparse unitary transformations need to be
stored. Furthermore, the application of the local orthogonal matrices
require only addition and subtraction, followed by a single
normalization to the entire vector at the end of each level. Because floating point
addition is typically quicker than floating point multiplication, the
application of $\mathcal{Q}_{\mathcal{A}_j}$ is more ideal
\cite{goldberg1991every}. Finally, we explicitly prove the complexity and storage of Algorithm 2.

\begin{theorem}
	Algorithm 2 has complexity $O(|E| \log |V|)$ and storage $O(|V|)$.
\end{theorem}

\begin{proof}
To verify that the algorithm is near-linear in complexity, we note
that the complexity of Algorithm~\ref{alg1} is
$O(|E| )$, and that the eigendecomposition requires at
most $|V|/2$ dictionary lookups or solutions to constant sized eigenvalue problems. The structure can have at most
$O( \log |V| )$ levels, each with at most half the vertices of the
previous level. This gives the overall computational cost of
$O(|E| \log ( |E| ) )$. Each unitary matrix $\mathcal{Q}_{\mathcal{A}}$
has at most $C_{\mathcal{A}} |V|$ non-zero entries. This results in an overall storage cost of $O(|V|)$.
\end{proof}

\subsection{An Example of Algorithm~\ref{alg2}}

In this section, on a small example we demonstrate the essential steps
in Algorithm~\ref{alg1} and Algorithm~\ref{alg2}.  Consider the graph
depicted in Fig.~\ref{fig:alg2-ex}.  The matching algorithm
Algorithm~\ref{alg1} starts with a graph $G$. The vertices of $G$ are
grouped by matching and they form a new graph $G_{\mathcal{A}_0}$. Next,
the graph
$G_{\mathcal{A}_1}$ is generated, and we continue until
a graph $G_{\mathcal{A}_2}$ with only one
vertex is generated.
Note that lines decorated with springs are those
chosen in the initial pairwise matching and the lines with small
segments are those added for isolated points.

After the matching is available, we compute the matrices
$Q_j$ (constructed in Algorithm~\ref{alg2}) as follows.
\begin{figure}
  \centering
  \begin{subfigure}[t]{5cm}
    \centering
    \caption*{$G$}
    \begin{tikzpicture}[scale=1.5,
      vertex/.style={circle,draw,thick,minimum size=5pt,inner sep=0pt},
      edge/.style={draw,thick,-},
      connected/.style={draw,thick,decorate,decoration={snake,
          segment length=2.5mm}},
      weakly connected/.style={draw,thick,decorate,decoration={ticks,
          pre length=1.1mm,post length=1.1mm,segment length=1.8mm,
          amplitude=1.1mm}}]
      \foreach \pos / \vertex in {(0,0)/1, (1,0)/2, (2,0)/3, (3,0)/4}
      \node[vertex,label={above:$\vertex$}] (\vertex) at \pos {};
      \foreach \pos / \vertex in {(1,-1)/5}
      \node[vertex,label={left:$\vertex$}] (\vertex) at \pos {};
      \foreach \pos / \vertex in {(2,-1)/6, (3,-1)/7, (0,-2)/8,
        (1,-2)/9}
      \node[vertex,label={below:$\vertex$}] (\vertex) at \pos {};
      \foreach \source / \dest in {1/2, 2/3, 2/5,
        3/4, 3/5, 3/6, 4/5, 4/6, 5/6, 5/8, 5/9, 6/7, 8/9}
      \path[edge] (\source) -- (\dest);
      \foreach \source / \dest in {1/2, 3/4, 6/7, 8/9}
      \path[connected] (\source) -- (\dest);
      \foreach \source / \dest in {5/8, 5/9}
      \path[weakly connected] (\source) -- (\dest);
    \end{tikzpicture}
  \end{subfigure}
  \qquad
  \begin{subfigure}[t]{2cm}
    \centering
    \caption*{$G_{\mathcal{A}_0}$}
    \begin{tikzpicture}[scale=1.5,
      vertex/.style={circle,draw,thick,minimum size=5pt,inner sep=0pt},
      edge/.style={draw,thick,-},
      connected/.style={draw,thick,decorate,decoration={snake,
          segment length=2.5mm}}]
      \foreach \pos / \vertex in {(0,0)/1, (1,0)/4}
      \node[vertex,label={above:$\vertex$}] (\vertex) at \pos {};
      \foreach \pos / \vertex in {(1,-1)/2, (0,-1)/3}
      \node[vertex,label={below:$\vertex$}] (\vertex) at \pos {};
      \foreach \source / \dest in {1/3, 1/4, 2/3, 2/4, 3/4}
      \path[edge] (\source) -- (\dest);
      \foreach \source / \dest in {1/3, 2/4}
      \path[connected] (\source) -- (\dest);
    \end{tikzpicture}
  \end{subfigure}
  \qquad
  \begin{subfigure}[t]{2cm}
    \centering
    \caption*{$G_{\mathcal{A}_1}$}
    \begin{tikzpicture}[scale=1.5,
      vertex/.style={circle,draw,thick,minimum size=5pt,inner sep=0pt},
      edge/.style={draw,thick,-},
      connected/.style={draw,thick,decorate,decoration={snake,
          segment length=2.5mm}}]
      \foreach \pos / \vertex in {(0,0)/1, (1,0)/2}
      \node[vertex,label={above:$\vertex$}] (\vertex) at \pos {};
      \foreach \source / \dest in {1/2}
      \path[edge] (\source) -- (\dest);
      \foreach \source / \dest in {1/2}
      \path[connected] (\source) -- (\dest);
    \end{tikzpicture}
  \end{subfigure}
  \qquad
  \begin{subfigure}[t]{0.5cm}
    \centering
    \caption*{$G_{\mathcal{A}_2}$}
    \begin{tikzpicture}[scale=1.5,
      vertex/.style={circle,draw,thick,minimum size=5pt,inner sep=0pt}]
      \foreach \pos / \vertex in {(0,0)/1}
      \node[vertex,label={above:$\vertex$}] (\vertex) at \pos {};
    \end{tikzpicture}
  \end{subfigure}
  \caption{Matching algorithm applied on a graph}
  \label{fig:alg2-ex}
\end{figure}
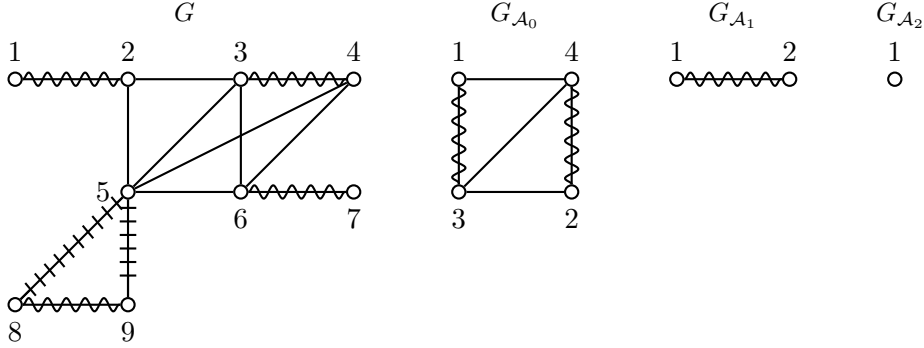
For the graph resulting from the
first matching $G_{\mathcal{A}_0}$, the orthgonal matrices
$Q_{a_i}$ are:
\begin{align*}
  & Q_{\{1,2\}} = Q_{\{3,4\}} = Q_{\{6,7\}} = \frac{1}{\sqrt2}
  \begin{bmatrix}
    1 & 1 \\ 1 & -1
  \end{bmatrix}, \\
  & Q_{\{5,8,9\}} =
  \begin{bmatrix}
    \frac{1}{\sqrt3} & \frac{1}{\sqrt2}  & \frac{1}{\sqrt2}  \\
    \frac{1}{\sqrt3} & 0                 & -\frac{1}{\sqrt2} \\
    \frac{1}{\sqrt3} & -\frac{1}{\sqrt2} & 0
  \end{bmatrix},
\end{align*}
and we have the following transformation matrix
{\tiny
\begin{equation*}
  Q_{\mathcal{A}_0} =
  \frac{1}{\sqrt2}
  \begin{bmatrix}
    1 & 1 & 0 & 0 & 0 & 0 & 0 & 0 & 0 \\
    1 & -1 & 0 & 0 & 0 & 0 & 0 & 0 & 0 \\
    0 & 0 & 1 & 1 & 0 & 0 & 0 & 0 & 0 \\
    0 & 0 & 1 & -1 & 0 & 0 & 0 & 0 & 0 \\
    0 & 0 & 0 & 0 & \sqrt{\frac23} & 0 & 0 & \sqrt{\frac23} & \sqrt{\frac23} \\
    0 & 0 & 0 & 0 & 0 & 1 & 1 & 0 & 0 \\
    0 & 0 & 0 & 0 & 0 & 1 & -1 & 0 & 0 \\
    0 & 0 & 0 & 0 & 1 & 0 & 0 & 0 & -1 \\
    0 & 0 & 0 & 0 & 1 & 0 & 0 & -1 & 0
  \end{bmatrix}.
\end{equation*}
}
The permutation which is applied to this marix is
\begin{equation*}
  \pi_0 =
  \begin{pmatrix}
    1 & 2 & 3 & 4 & 5 & 6 & 7 & 8 & 9 \\
    1 & 5 & 4 & 8 & 3 & 2 & 6 & 7 & 9
  \end{pmatrix}.
\end{equation*}
and as a result we get
{\tiny
\begin{equation*}
    Q_0 =
  \frac{1}{\sqrt2}
  \begin{bmatrix}
    1 & 1 & 0 & 0 & 0 & 0 & 0 & 0 & 0 \\
    0 & 0 & 0 & 0 & 0 & 1 & 1 & 0 & 0 \\
    0 & 0 & 0 & 0 & \sqrt{\frac23} & 0 & 0 & \sqrt{\frac23} & \sqrt{\frac23}   \\
    0 & 0 & 1 & 1 & 0 & 0 & 0 & 0 & 0 \\
    1 & -1 & 0 & 0 & 0 & 0 & 0 & 0 & 0 \\
    0 & 0 & 0 & 0 & 0 & 1 & -1 & 0 & 0 \\
    0 & 0 & 0 & 0 & 1 & 0 & 0 & 0 & -1 \\
    0 & 0 & 1 & -1 & 0 & 0 & 0 & 0 & 0 \\
    0 & 0 & 0 & 0 & 1 & 0 & 0 & -1 & 0
  \end{bmatrix}.
\end{equation*}
}
For the graph resulting from the
second matching, $G_{\mathcal{A}_1}$, the matrix $Q_1$ has a block diagonal form,  $Q_1 = \operatorname{diag}(\pi_1Q_{\mathcal{A}_1},
\pi_1Q_{\mathcal{A}_1}, 1)$, where
\begin{equation*}
  Q_{\mathcal{A}_1} = 1
  \begin{bmatrix}
    1 & 0 & 1 & 0 \\
    0 & 1 & 0 & 1 \\
    1 & 0 & -1 & 0 \\
    0 & 1 & 0 & -1
  \end{bmatrix}, \quad
  \pi_1 = \operatorname{id}_{\{1,2,3,4\}}.
\end{equation*}
Therefore, for $G_{\mathcal{A}_2}$, we get
{\tiny
\begin{equation*}
  Q_1 =\frac{1}{\sqrt2}
  \begin{bmatrix}
    1 & 0 & 1 & 0 & 0 & 0 & 0 & 0 & 0 \\
    0 & 1 & 0 & 1 & 0 & 0 & 0 & 0 & 0 \\
    1 & 0 & -1 & 0 & 0 & 0 & 0 & 0 & 0 \\
    0 & 1 & 0 & -1 & 0 & 0 & 0 & 0 & 0 \\
    0 & 0 & 0 & 0 & 1 & 0 & 1 & 0 & 0 \\
    0 & 0 & 0 & 0 & 0 & 1 & 0 & 1 & 0 \\
    0 & 0 & 0 & 0 & 1 & 0 & -1 & 0 & 0 \\
    0 & 0 & 0 & 0 & 0 & 1 & 0 & -1 & 0 \\
    0 & 0 & 0 & 0 & 0 & 0 & 0 & 0 &   \sqrt2
  \end{bmatrix},\;\;
 Q_2 = 
  \frac{1}{\sqrt2}
  \begin{bmatrix}
    1 & 1 & 0 & 0 & 0 & 0 & 0 & 0 & 0 \\
    1 & -1 & 0 & 0 & 0 & 0 & 0 & 0 & 0 \\
    0 & 0 & 1 & 1 & 0 & 0 & 0 & 0 & 0 \\
    0 & 0 & 1 & -1 & 0 & 0 & 0 & 0 & 0 \\
    0 & 0 & 0 & 0 & 1 & 1 & 0 & 0 & 0 \\
    0 & 0 & 0 & 0 & 1 & -1 & 0 & 0 & 0 \\
    0 & 0 & 0 & 0 & 0 & 0 & 1 & 1 & 0 \\
    0 & 0 & 0 & 0 & 0 & 0 & 1 & -1 & 0 \\
    0 & 0 & 0 & 0 & 0 & 0 & 0 & 0 & \sqrt2
  \end{bmatrix}.
\end{equation*}
}
Finally, the matrices
$Q_0$, $Q_1$, and $Q_2$ are all constructed and ready to use
for approximation of signals as in Algorithm~\ref{alg2}.

\subsection{A nonlinear version of Algorithm~\ref{alg2}}
\label{nonlinear-alg2} In this subsection, we consider a nonlinear version of
Algorithm~\ref{alg2}, motived by ideas in nonlinear approximation
found in DeVore~\cite[Section~3]{devore1998}.

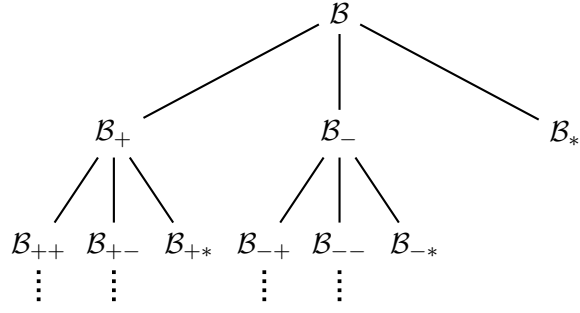
\begin{figure}[!htb]
  \centering
  \begin{tikzpicture}[thick,
    level 1/.style={level distance=1.6cm,sibling distance=3cm},
    level 2/.style={level distance=1.5cm,sibling distance=1cm},
    level 3/.style={level distance=.75cm,
      every child/.style={
        edge from parent/.style={draw,dotted,line width=0.5mm}}}]
    \node {$\mathcal{B}$}
    child {
      node {$\mathcal{B}_+$}
      child {
        node {$\mathcal{B}_{++}$}
        child
      }
      child {
        node {$\mathcal{B}_{+-}$}
        child
      }
      child {
        node {$\mathcal{B}_{+*}$}
      }
    }
    child {
      node {$\mathcal{B}_-$}
      child {
        node {$\mathcal{B}_{-+}$}
        child
      }
      child {
        node {$\mathcal{B}_{--}$}
        child
      }
      child {
        node {$\mathcal{B}_{-*}$}
      }
    }
    child {
      node {$\mathcal{B}_*$}
    };
  \end{tikzpicture}
  \caption{Hierarchy of orthonormal bases formed in constructing the
    frequency domain.}
  \label{fig:wavelet}
\end{figure}

As it is immediately seen, Algorithm \ref{alg2} produces not only a
frequency basis, but a hierarchical structure which defines a
hierarchy of bases (see Figure~\ref{fig:wavelet}). The collection of
all such bases is a frame for the underlying space~$\R{n}$. For a
given element $f\in \R{n}$ we choose expansion in a basis from this
frame which has fast decay in  the absolute values of the values of the
coefficients in the expansion. We now explain this in more detail and refer
to~\cite{devore1998} for more general cases.

Let $\mathcal{B} := \left\{\bm{e}_1, \bm{e}_2, \dots,
\bm{e}_n\right\}$ be the natural basis of $\mathbb{R}^n$. We can write
the matrix $(\pi Q_{\mathcal{A}})^*$ as
\begin{equation}\label{eq-def-basis}
  (\pi Q_{\mathcal{A}})^* =
  \begin{bmatrix}
    \bm{q}_{+,1} & \cdots & \bm{q}_{+,n_1} &
    \bm{q}_{-,1} & \cdots & \bm{q}_{-,n_1} &
    \bm{q}_{*,1} & \cdots & \bm{q}_{*,n-2n_1}
  \end{bmatrix}.
\end{equation}
We further define
\begin{alignat*}{2}
  & \widetilde{\bm{q}}_{s+} && :=
  \begin{bmatrix}
    \bm{q}_{s+,1} & \cdots & \bm{q}_{s+,n_{|s|+1}}
  \end{bmatrix}, \\
  & \widetilde{\bm{q}}_{s-} && :=
  \begin{bmatrix}
    \bm{q}_{s-,1} & \cdots & \bm{q}_{s-,n_{|s|+1}}
  \end{bmatrix}, \\
  & \widetilde{\bm{q}}_{s*} && :=
  \begin{bmatrix}
    \bm{q}_{s*,1} & \cdots & \bm{q}_{s*,n_{|s|}-2n_{|s|+1}}
  \end{bmatrix},
\end{alignat*}
where $s\in \{+,-,\}^j, j = 0,\dots,J$, and $s+$, $s-$ and $s*$ are the strings obtained by appending `$+$', `$-$' and `$*$' to $s$,
respectively. Given the above notation, one can rewrite
\eqref{eq-def-basis} as
\begin{equation*}
  (\pi Q_{\mathcal{A}})^* =
  \begin{bmatrix}
    \widetilde{\bm{q}}_+ & \widetilde{\bm{q}}_- & \widetilde{\bm{q}}_*
  \end{bmatrix}.
\end{equation*}
Let $\operatorname{cols}(A)$ be the set of column vectors of a matrix $A$, and define
\begin{equation*}
  \mathcal{B}_{s+} := \operatorname{cols}(\widetilde{\bm{q}}_{s+}), \quad
  \mathcal{B}_{s-} := \operatorname{cols}(\widetilde{\bm{q}}_{s-}), \quad
  \mathcal{B}_{s*} := \operatorname{cols}(\widetilde{\bm{q}}_{s*}).
\end{equation*}
We arrange $\mathcal{B}_+$, $\mathcal{B}_-$ and $\mathcal{B}_*$ as the
children of the root node $\mathcal{B}$ in the tree. In a similar
fashion, we recursively define
\begin{equation*}
  \begin{bmatrix}
    \widetilde{\bm{q}}_{s+} & \widetilde{\bm{q}}_{s-} &
    \widetilde{\bm{q}}_{s*}
  \end{bmatrix} := \widetilde{\bm{q}}_s (\pi_iQ_{\mathcal{A}_i})^*,
  \quad \text{for all } s\in  \{+,-\}^j, \ j = 0,\dots,J.
\end{equation*}
Clearly, $\mathcal{B}_{s+}  \cup \mathcal{B}_{s-} \cup \mathcal{B}_{s*}$ spans the same subspace of $\mathbb{R}^n$ as
$\mathcal{B}_s$. In addition, given $j \in \{0,1,\dots,J\}$, the set
of vectors
\begin{equation*}
  \bigcup_{s\in \{+,-,\}^j} (\mathcal{B}_{s+} \cup \mathcal{B}_{s-})
  \,\cup\, \bigcup_{i=0}^j\bigcup_{s\in \{+,-\}^i}\mathcal{B}_{s*}
\end{equation*}
are exactly the columns of
$\left(\prod_{i=0}^j\mathcal{Q}_{j-i}\right)^*$. 

The basis produced by Algorithm~\ref{alg2} is the specific case where
$j = J$. However, we can choose an adaptive basis by taking the union
of various sets $\mathcal{B}_\cdot$. Following~\cite{devore1998}, for
a given signal $f$ on the graph, we choose the basis
$\mathcal{B}^{\text{a}}$ adaptively so that the
$\ell^1$-norm of the coefficients of $f$ in this basis
\begin{equation*}
  N(f,\mathcal{B}^{\text{a}}) := \sum_{\bm{q}\in\mathcal{B}^{\text{a}}}
  \lvert(f, \bm{q})\rvert
\end{equation*}
is minimized. We note that this guarantees coefficient decay at least
as fast of that of the basis produced by Algorithm \ref{alg2}. We
describe the minimization procedure formally in Algorithm~\ref{alg3}.

\begin{algorithm}
  \caption{Find the Basis that Minimizes the $\ell^1$-norm of
    Coefficients}
  \label{alg3}
  \begin{algorithmic}[]

    \State For each
    $s\in \{+,-\}^{J+1}$, set
    $\mathcal{B}_s^{\text{a}} \leftarrow \mathcal{B}_s$.
    
    \For{$j = J, J-1, \dots, 0$}
    \For{$s\in \{+,-\}^j$}
    \State $\mathcal{B}_s^{\text{a}} \leftarrow
    \argmin \left\{ N(f,\mathcal{B}') \mid
      \mathcal{B}' \in \left\{ \mathcal{B}_s,
        \mathcal{B}_{s+}^{\text{a}} \cup \mathcal{B}_{s+}^{\text{a}}
        \cup \mathcal{B}_{s*} \right\} \right\}$
    \EndFor
    \EndFor

    \State Output the basis $\mathcal{B}^{\text{a}} \leftarrow \mathcal{B}_{\emptyset}^{\text{a}} $
  \end{algorithmic}
\end{algorithm}

In addition, we have the following complexity and storage guarantees.

\begin{theorem}
  Algorithm~\ref{alg3} has complexity
  $O(\lvert V\rvert\log \lvert V\rvert)$ and storage $\lvert V\rvert$.
\end{theorem}

\begin{proof}
  The depth of the hierarchy is at most $\log\lvert V\rvert$, within
  each hierarchy, a total number of $\lvert V\rvert$ inner products
  need to be computed for the terms $N(f,\mathcal{B}')$. Therefore the
  complexity of the algorithm is $O(\lvert V\rvert\log \lvert
  V\rvert)$. To store the basis chosen, we only need to store one
  boolean value for each element in $\{+,-\}^j$ ($j = 0, 1, \dots,
  J$): $1$ if $\mathcal{B}_s$ is chosen, or $0$ if
  $\mathcal{B}_{s+}^{\text{a}} \cup \mathcal{B}_{s+}^{\text{a}} \cup
  \mathcal{B}_{s*}$ is chosen. There are $2^{J+1} - 1$ such nodes, and
  we have $2^{J+1} \leq \lvert V\rvert$. This gives at most $\lvert
  V\rvert$ storage.
\end{proof}

\section{Approximation of Smooth Signals}

In this section, we analyze the decay of smooth signals in our frequency domain. In particular, let
$$|u|_{G,\infty} := \max_{(i,j) \in E} |u_i - u_j| .$$
Recall, the columns of $\mathcal{Q^*}$ form the orthogonal basis of our frequency domain. Let $q_i = \prod_{j=0}^J \mathcal{Q}^*_{j} e_i$. For each basis element, we can associate the label
$$\phi(i) := \max \{ j | n_j \ge i \}  .$$
This naturally gives rise to a partition
$$\Phi_j := \{ i | \phi(i) = j \}, \quad \bigcup_{j=0}^J  \Phi_j = [n], \quad  \Phi_{j_1} \cap \Phi_{j_2} = \emptyset, j_1 \ne j_2.$$

First, we give the following lemma regarding the approximation of smooth vectors by the constant vector $\c1 := (1,...,1)^T$.

\begin{lemma} \label{lm:aprx}
Let $u \in \R{m}$, $m>1$. If
$$| u_{i} - u_{i+1} | \le \alpha \quad \text{ for } i = 1,...,m-1,$$
then
$$\|u \|^2 - \frac{ \langle 1_m , u \rangle^2}{\langle 1_m, 1_m \rangle} \le  \frac{\alpha^2 m^3}{12}.$$
\end{lemma}

\begin{proof}
Let
$$ u = \sum_{i=1}^m \alpha_i \hat 1_i, \qquad \hat 1_i(j) := \left\{ \begin{matrix} 1 \quad j \le i \\ 0 \quad j>i \end{matrix}   \right.$$
Then
\begin{align*}
\|u \|^2 - \frac{ \langle 1_m , u \rangle^2}{\langle 1_m, 1_m \rangle} &= \left\langle  \sum_{i=1}^m \alpha_i \hat 1_i,  \sum_{i=1}^m \alpha_i \hat 1_i \right\rangle - \frac{1}{m} \left \langle 1_m, \sum_{i=1}^m \alpha_i \hat 1_i \right \rangle^2  \\
&= \sum_{i=1}^m i \alpha_i^2 + 2 \sum_{i=1}^{m-1} \sum_{j=i+1}^{m} i \alpha_i \alpha_j - \frac{1}{m} \left[ \sum_{i=1}^m i^2 \alpha_i^2 + 2 \sum_{i=1}^{m-1} \sum_{j=i+1}^{m} ij \alpha_i \alpha_j \right] \\
&= \frac{1}{m} \left[ \sum_{i=1}^m (m-i)i \alpha_i^2 + 2 \sum_{i=1}^{m-1} \sum_{j=i+1}^{m} (m-j)i \alpha_i \alpha_j \right] \\
&= \frac{1}{m}  \left[ \sum_{i=1}^{m-1} (m-i)i \alpha_i^2 + 2 \sum_{i=1}^{m-2} \sum_{j=i+1}^{m-1} (m-j)i \alpha_i \alpha_j \right]. \\
\end{align*} 
Noting that we have $|\alpha_i| \le \alpha $, $i = 1,...,m-1$, the result quickly follows:
\begin{align*}
\|u \|^2 - \frac{ \langle 1_m , u \rangle^2}{\langle 1_m, 1_m \rangle}  &\le 
\frac{\alpha^2}{m}  \left[ \sum_{i=1}^{m-1} (m-i)i  + 2 \sum_{i=1}^{m-2} \sum_{j=i+1}^{m-1} (m-j)i \right] \\
&= \frac{\alpha^2}{m} \left[ \frac{1}{6}(m-1)m(m+1) + \frac{1}{12}(m-2)(m-1)m(m+1) \right] \\
&= \frac{\alpha^2}{12}(m-1)m(m+1) \le \frac{\alpha^2 m^3}{12}.
\end{align*}
\end{proof}

We have the following result regarding the expansion of a smooth signal in the basis $\{q_i\}_{i=1}^n$.

\begin{lemma} \label{lm:exp}
$$\sum_{i \in \Phi_k} \langle u ,  q_i \rangle^2 \le\frac{n}{12} \bigg[\frac{C_{\mathcal{A}}^3}{2}\bigg]^{k+1}  |u|^2_{G,\infty}.$$
\end{lemma}

\begin{proof}
First, we will treat $\Phi_0$. The elements of $\Phi_0$ are precisely those which correspond to the sets $V_{-}$ and $V_*$ of $G$. We will bound the sum of the quantities $[\mathcal{Q}_0 u](i)^2$, $i \in \Phi_0$. By Lemma \ref{lm:aprx}, we have
\begin{align*}
\sum_{i \in \Phi_0} [\mathcal{Q}_0 u](i)^2 &= \sum_{a_\ell \in \mathcal{A}_0} ||u_{a_\ell} ||^2 - \frac{ \langle 1_{|a_\ell|} , u_{a_\ell} \rangle }{\langle 1_{|a_\ell|} ,1_{|a_\ell|} \rangle } \\
&\le \sum_{a_\ell \in \mathcal{A}_{0}} \frac{ |u|^2_{G,\infty} | a_\ell|^3}{12} \\
&\le \frac{C_{\mathcal{A}}^3}{24}  n |u|^2_{G,\infty} .
\end{align*}

In addition, because each $\mathcal{Q}_j$ is unitary and is block diagonal with respect to the partition $\Phi_0$ and $[n] \backslash \Phi_0$, this implies that
$$\sum_{i \in \Phi_0} \langle u ,  q_i \rangle^2 \le \frac{C_{\mathcal{A}}^3}{24} n |u|^2_{G,\infty}.$$

The vector $u_+$ satisfies
$$|u_+|_{G_{\mathcal{A}_0}, \infty} \le |u|_{G,\infty} +  \frac{2}{\sqrt{C_{\mathcal{A}}}}  \sum_{i = 1}^{C_{\mathcal{A}} -1} |u|_{G,\infty} i \le C_{\mathcal{A}}^{3/2} |u|_{G,\infty}.$$

Now, proceeding by induction, for $\Phi_k$ we will repeat the same procedure, except we will have
$$|v_+|_{G_{\mathcal{A}_{k-1}}, \infty} \le C_{\mathcal{A}}^{3k/2} |u|_{G,\infty}.$$
Using Lemma \ref{lm:aprx} and noting that $|\mathcal{A}_k| \le 2^{-(k+1)} n$, we obtain the bound
\begin{align*}
\sum_{i \in \Phi_k} \langle u ,  q_i \rangle^2 &\le  | \mathcal{A}_k|  \bigg[ \frac{C_{\mathcal{A}}^{3k} |u|^2_{G,\infty} C_{\mathcal{A}}^3}{12}  \bigg] \le \frac{n}{12} \bigg[\frac{C_{\mathcal{A}}^3}{2}\bigg]^{k+1} |u|^2_{G,\infty}.
\end{align*}
\end{proof}

Based on Lemma \ref{lm:exp}, we can now give theoretical bounds on the quality of the $k$-term approximation $\widehat{u}_{\sigma_k}$, which only depends on the smoothness of the original signal and the quality of aggregation produced. We have the following theorem.

\begin{theorem}
$$\| u- \mathcal{Q}^* \widehat {u_{\sigma_k}}\|^2 \le \frac{C^3_{\mathcal{A}}n}{12} \bigg[ \frac{n}{k}\bigg]^{\log_2 C^3_{\mathcal{A}} -1}  |u|^2_{G,\infty}.$$
\end{theorem}

\begin{proof}
\begin{align*}
\| u- \mathcal{Q}^* \widehat {u_{\sigma_k}}\|^2 &\le \sum_{i=0}^{\lceil \log_2 \frac{n}{k} \rceil -1} \frac{n}{12} \bigg[\frac{C_{\mathcal{A}}^3}{2}\bigg]^{i+1} |u|^2_{G,\infty} \\
&\le \frac{n}{6}  \bigg[\frac{C_{\mathcal{A}}^3}{2}\bigg]^{\lceil \log_2 \frac{n}{k} \rceil } |u|^2_{G,\infty} \\
&\le \frac{n}{6}  \bigg[\frac{C_{\mathcal{A}}^3}{2}\bigg]^{ \log_2 \frac{n}{k} +1 } |u|^2_{G,\infty} \\
&= \frac{C^3_{\mathcal{A}}n}{12}  \bigg[\frac{C_{\mathcal{A}}^3}{2}\bigg]^{\log_2 \frac{n}{k} } |u|^2_{G,\infty} \\
&=\frac{C^3_{\mathcal{A}}n}{12} \bigg[ \frac{n}{k}\bigg]^{\log_2 C^3_{\mathcal{A}} -1}  |u|^2_{G,\infty}.
\end{align*}
\end{proof}

This proves that for sufficiently smooth signals and perfect matchings, the decay in the coefficients of our frequency domain is of order at least $k^{-2}$.

\section{Results and Discussion} \label{sec:num}

We now consider how the complexity, storage costs and performance of
Algorithm~\ref{alg2} for several examples.

We note that spectral-based techniques require at least $O( k
|E|\log^{1/2} | V | )$ computations and $O(k |V| )$ storage to compute
a $k$-dimensional low-frequency subspace, and $O(|V| |E| \log^{1/2}
|V|)$ computations and $O(|V|^2)$ storage to create a full frequency
domain \cite{Cohen:2014:SSL:2591796.2591833}. The latter procedure is
untenable for even moderately large graphs. Algorithm \ref{alg2} is
superior in complexity and storage for any choice of $k = \omega (\log^{1/2}
|V|)$, and it produces an entire frequency domain rather than a
$k$-dimensional subspace. See Table \ref{tbl1}.

In addition, Algorithm~\ref{alg2} is applicable to
low-pass data compression. For a given sparse network $G = (V,E)$, $n
= |V|$, $|E| = |V| \log |V|$, suppose we have $m$ smooth signals $\{
u^{(i)} \}_{i = 1}^m \subset \R{n}$, with $m = n^\alpha$, $0 < \alpha
< 1$. The current storage requirements for the network and the data
combined are $O(|V|^{1+\alpha})$. Suppose we wish to compress this
data so that the overall storage requirements are $O(|E|)$. Using a
$k$-frequency low-pass filter for each $u^{(i)}$, the $m$ smooth
signals can be compressed and stored using $O(k |V|^\alpha )$ memory,
not including the frequency domain and network storage. Using a
spectral approach, we can only perform $O( \log |V| )$-frequency
low-pass filtering. On the other hand, using Algorithm \ref{alg2}, we
can compute the full frequency domain with only $O(|V|)$
storage. Because the network is sparse, we can safely perform $|V|^{1-
	\alpha}$-frequency low-pass filtering and still have storage
requirements for the signals that are less than that of the
graph. Therefore, although the quality of low-pass spectral
frequencies is slightly superior to those of Algorithm \ref{alg2}, the
large number of frequencies that Algorithm \ref{alg2} provides results
in improved and more robust compression.

 \begin{table}[]
\caption{Complexity and Storage Costs of Algorithm \ref{alg2} and
  $k$-Frequency Spectral Subspace}
\begin{center}
\begin{tabular}{ l | l | l }
Costs & $k$-Spectral &  Algorithm \ref{alg2}  \\ \hline
 Make Frequency Domain & $O(k |E| \log^{1/2} |V|)$ & $O(|E| \log |V|)$ \\
Map to/from Domain & $O(k |V|)$ &  $O(|V|)$ \\ 
Store Frequency Domain & $O(k |V|)$ & $O(|V|)$ \\
\end{tabular}
\end{center}
\label{tbl1}
\end{table}

To investigate the performance of Algorithm \ref{alg2} and \ref{alg3} we consider two
simulations on real-world networks.

\subsection{Compression of monthly precipitation data}
The first example is on compressing the monthly precipitation data
over Punjab, India in the rectangle
$[29^{\circ}N, 33^{\circ}N]\times [73.25^{\circ}E,77.5^{\circ}E]$. The
original data is on $(512\times 512)$ grid. We compress the data using
our compression algorithm and show how the recovered data is like
compared to the original plot (Figure \ref{fig:punjab}).  For
compressing the precipitation data, lattice graph of size
$512\times512$ is used. Figure \ref{fig:punjab} shows the plotted data
over the 12 months of an year. For each month, the first plot is a
plot of the original data, while the second plot is data recovered
after compression using Algorithm \ref{alg2}. The third (most right)
image shows the data recovered from the adaptive compression algorithm
described in \S\ref{nonlinear-alg2}. The filtering threshold for
compression is chosen to be 10,000 (out of about 260,000). In Figure
\ref{fig:punjab-levels} we plot the relative error of the compression
for different filtering threshold $k$.

\newcommand{\plotmonth}[2]{
  \begin{subfigure}{.45\textwidth}
    \centering
    \includegraphics[width=0.3\textwidth]{agp1980_#1}~
    \includegraphics[width=0.3\textwidth]{agp1980_#1_1}~
    \includegraphics[width=0.3\textwidth]{agp1980_#1_2}
    \caption{#2}
  \end{subfigure}
}
\begin{figure}
  \centering
  \plotmonth{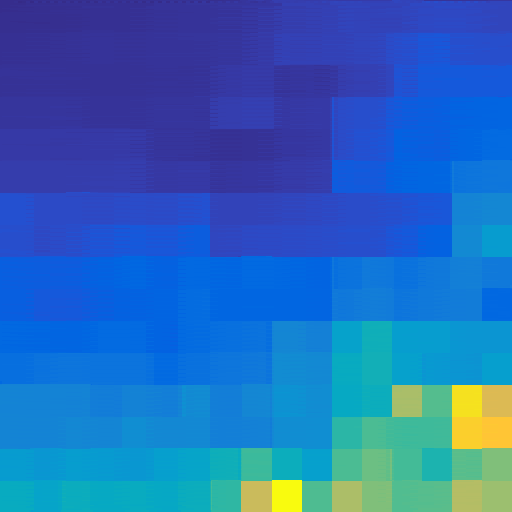}{January}~\plotmonth{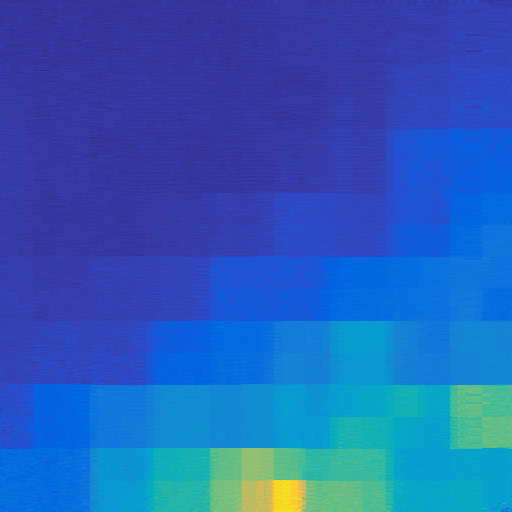}{February}\par
  \plotmonth{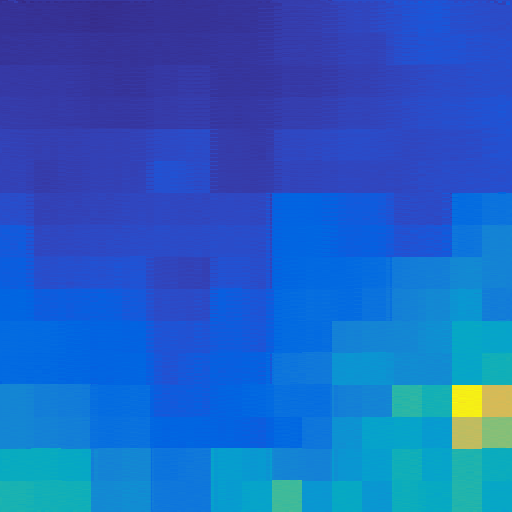}{March}~\plotmonth{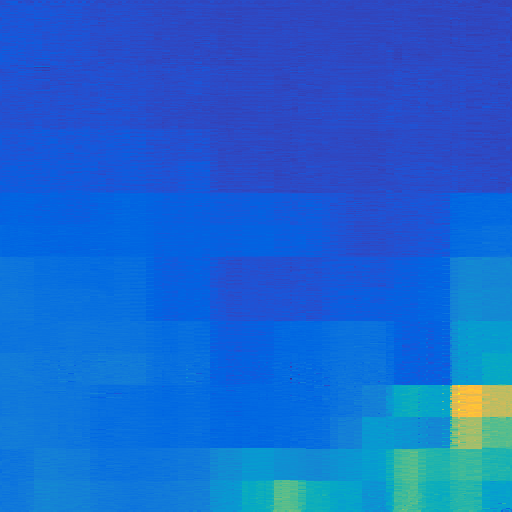}{April}\par
  \plotmonth{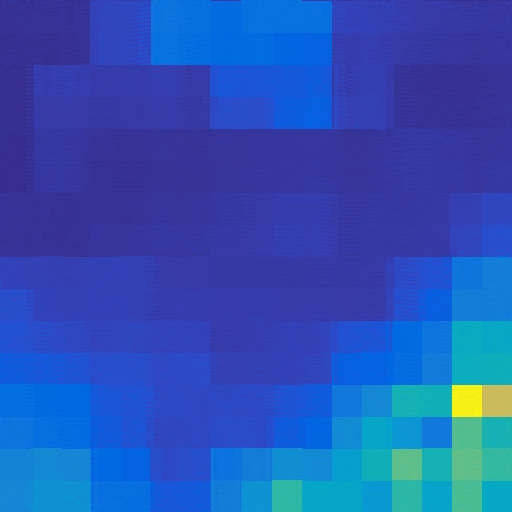}{May}~\plotmonth{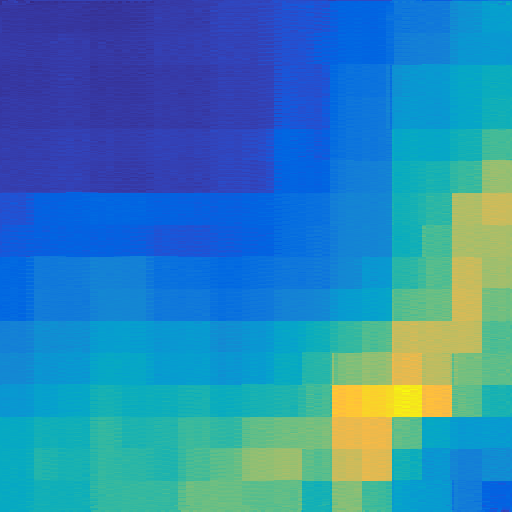}{June}\par
  \plotmonth{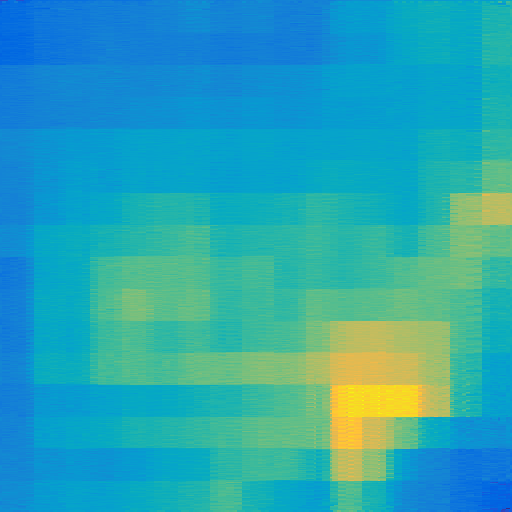}{July}~\plotmonth{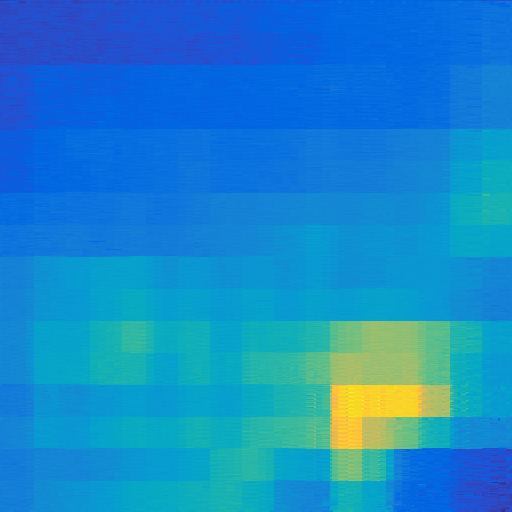}{August}\par
  \plotmonth{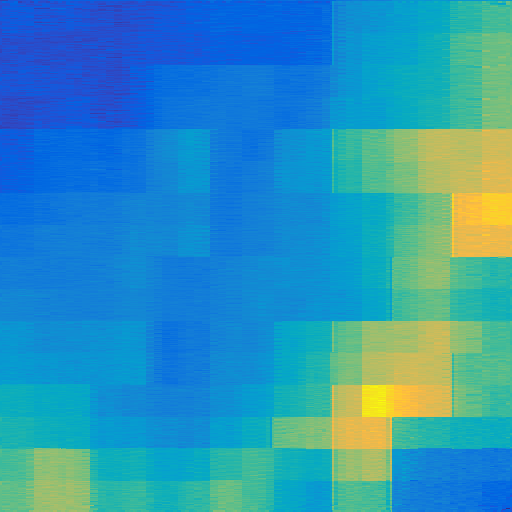}{September}~\plotmonth{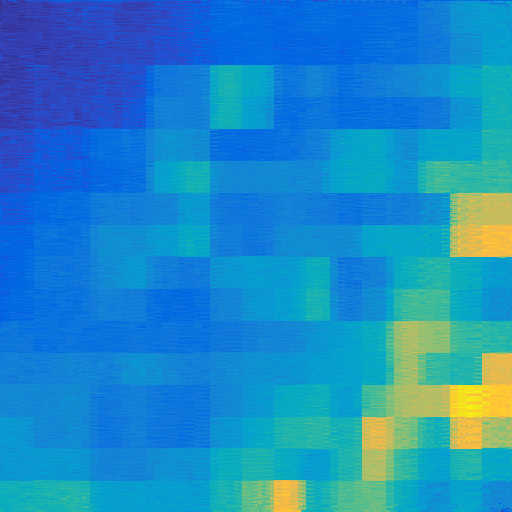}{October}\par
  \plotmonth{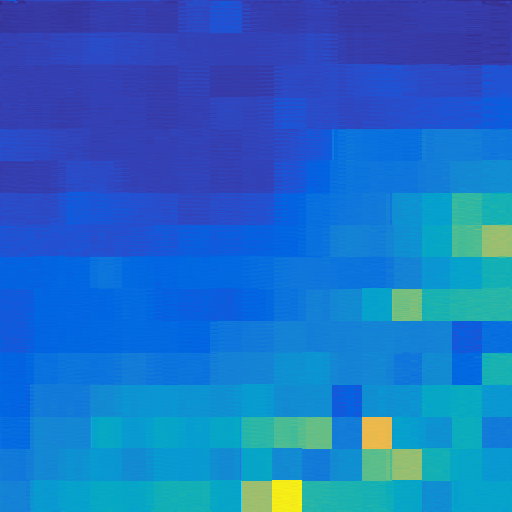}{November}~\plotmonth{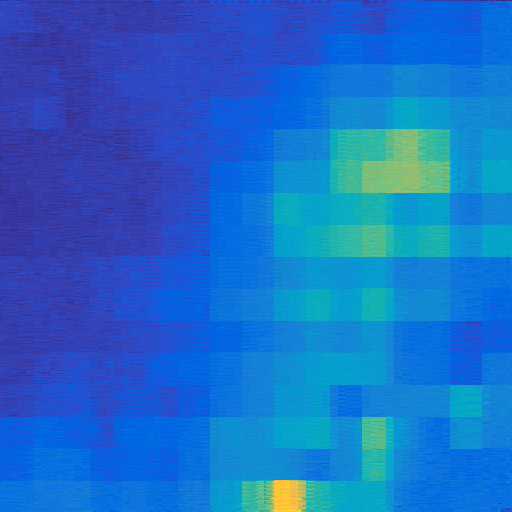}{December}\par
  \caption{Precipitation data over 12 months from Punjab, India, in
    year 1980. The data were compressed and decompressed using both
    the original and adaptive compression algorithms.}
  \label{fig:punjab}
\end{figure}

\newcommand{\plotmonthlevels}[2]{
  \begin{subfigure}{.31\textwidth}
    \centering
    \includegraphics[width=\textwidth]{agp1980_#1_levels}
    \caption{#2}
  \end{subfigure}
}
\begin{figure}
  \centering
  \plotmonthlevels{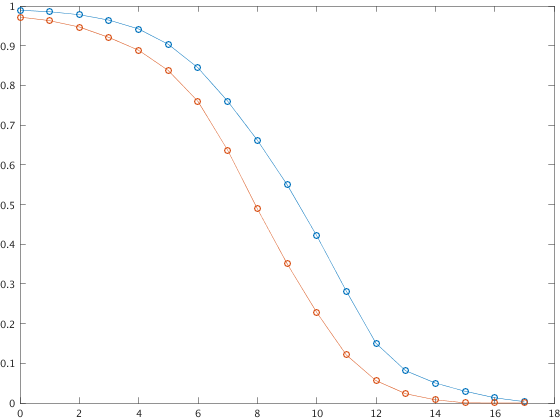}{January}
  \plotmonthlevels{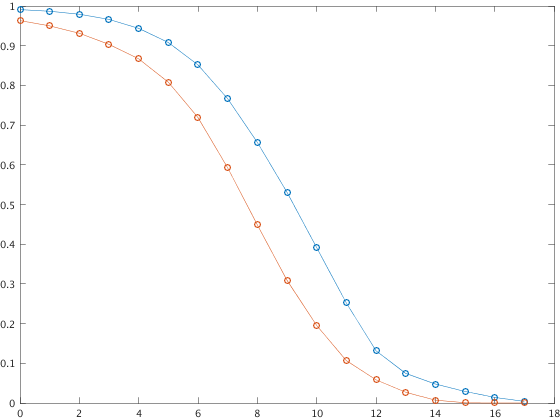}{February}
  \plotmonthlevels{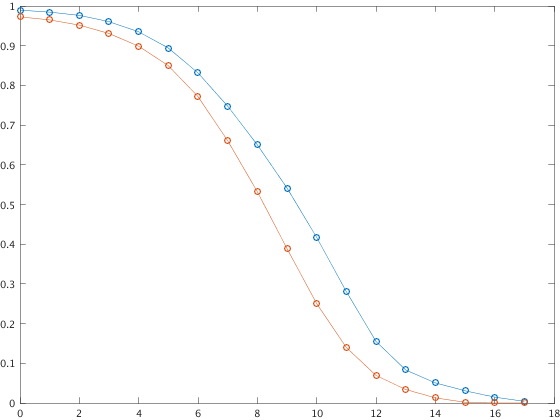}{March}\par
  \plotmonthlevels{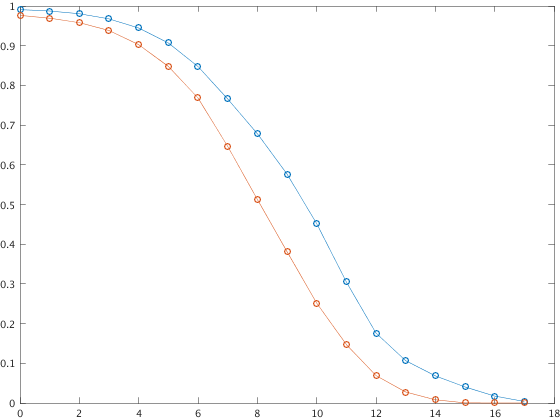}{April}
  \plotmonthlevels{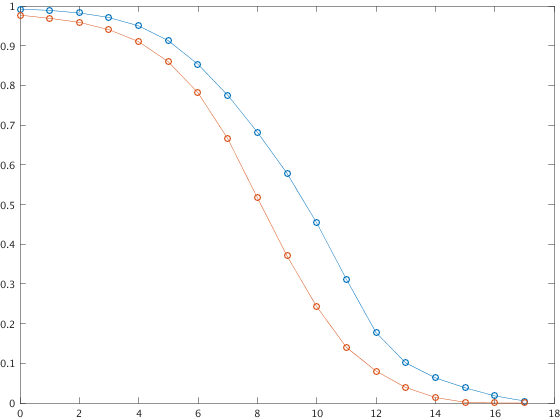}{May}
  \plotmonthlevels{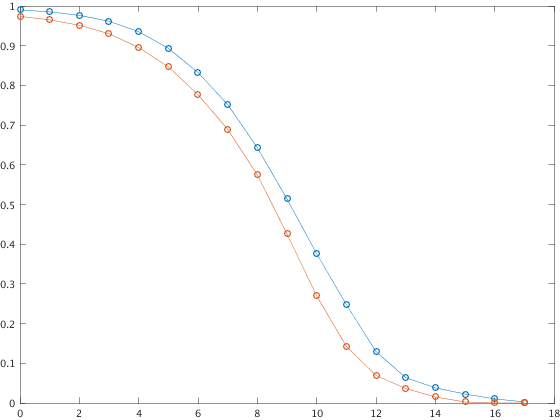}{June}\par
  \plotmonthlevels{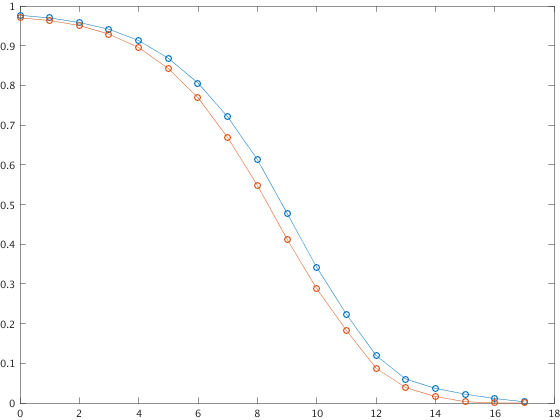}{July}
  \plotmonthlevels{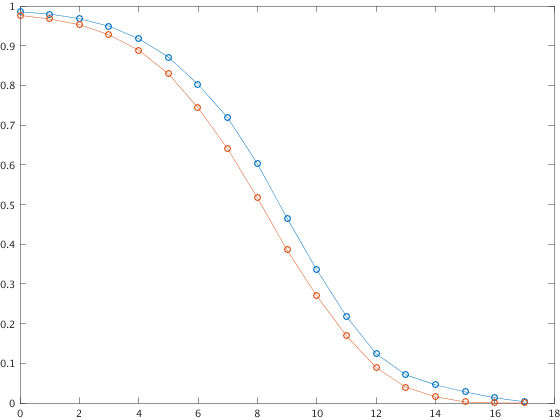}{August}
  \plotmonthlevels{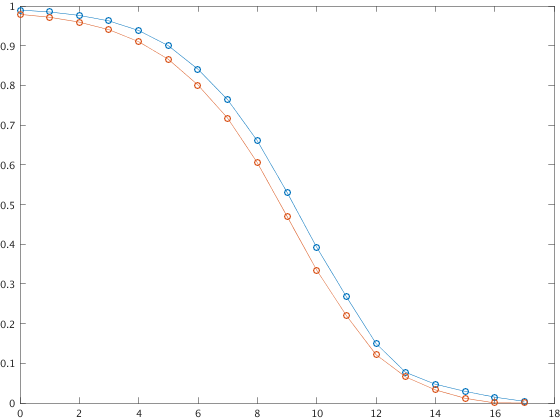}{September}\par
  \plotmonthlevels{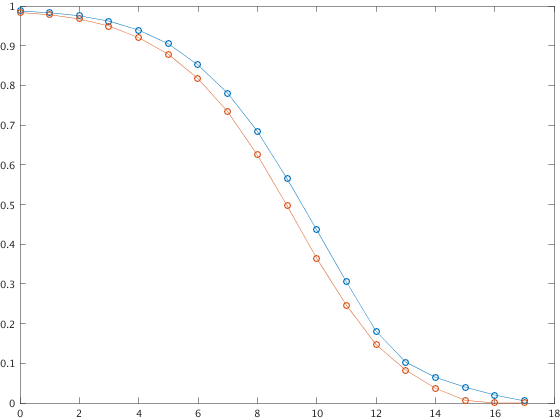}{October}
  \plotmonthlevels{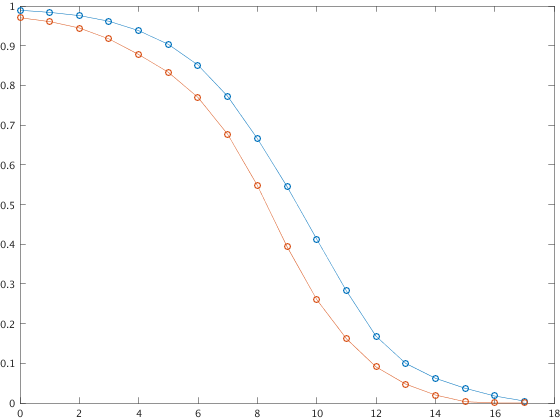}{November}
  \plotmonthlevels{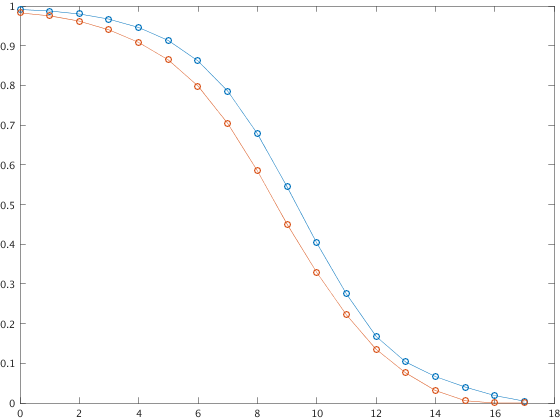}{December}\par
  \caption{Plots of relative error $\lVert u-v \rVert/\lVert u\rVert$
    on the Punjab precipitation data. The horizontal axis is
    $\log_2 k$, $k$ being the filtering threshold. For each plot,
    blue line is Algorithm \ref{alg2}, red line is adaptive
    compression.}
  \label{fig:punjab-levels}
\end{figure}

\subsection{Low pass filtering on average annual precipitation}
Another example, is on low-pass filtering on average annual
precipitation data for the contiguous United States in Figure
\ref{fig:states}.  The spectral filter is mildly, but not notably
superior to Algorithm \ref{alg2}.
\begin{figure}[]
\center
\begin{subfigure}{0.3\textwidth}\caption{Precipitation Data}{\includegraphics[width = 1.6 in,height = 1.5 in]{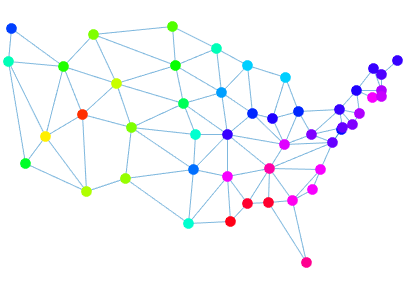}} \quad
\end{subfigure}\begin{subfigure}{0.3\textwidth}\caption{Spectral Filter}{\includegraphics[width = 1.6 in,height = 1.5 in]{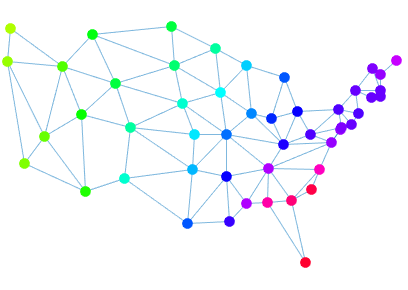}} \quad
\end{subfigure}
\begin{subfigure}{0.3\textwidth}\caption{Algorithm \ref{alg2} Filter}{\includegraphics[width = 1.6 in,height = 1.5 in]{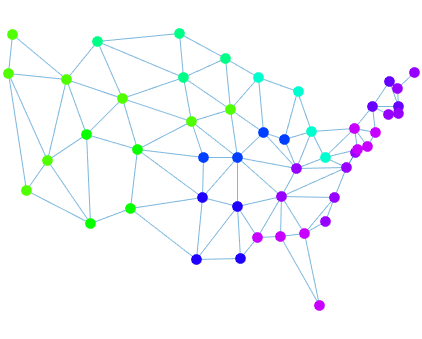}}
\end{subfigure}
  \caption{Average annual precipitation for states in the
contiguous United States, in Figure (a). Spectral filter of
precipitation data, using the three minimal, nontrivial
eigenvectors, in Figure (b). Algorithm \ref{alg2} filter, using
three component compression, in Figure (c). For both filters,
signal was pre-processed with centering and normalization. Source
\cite{usdata}.}  \label{fig:states} \end{figure}

\section*{Acknowledgments}
The work of W.~Xu and L.~Zikatanov was supported in part by NSF through grants
DMS-1522615 and DMS-1720114.  The authors are grateful to Louisa
Thomas for greatly improving the style of presentation. The authors
also thank Madeline Nyblade and Dr. Tess Russo for the invaluable help
with the monthly precipitation data in Punjab, India.


\end{document}